\documentclass{article}

\usepackage[margin=1in]{geometry}

\usepackage{amssymb}
\usepackage{amsmath}
\usepackage{amsthm}
\usepackage{enumitem}
\usepackage{bm}
\usepackage{tikz}
\usetikzlibrary{arrows.meta}
\usetikzlibrary{decorations} 
\usetikzlibrary{shapes}
\usepackage[margin=2cm]{caption}
\usepackage{subcaption}
\usepackage{titlesec}
\usepackage[backref]{hyperref}
\usepackage{graphicx}
\usepackage{xcolor}
\usepackage{thmtools}
\usepackage{thm-restate}
\usepackage{cleveref}
\usepackage{lmodern}
\usetikzlibrary{arrows.meta,backgrounds,calc}

\newtheorem{theorem}{Theorem}[section]

\newtheorem{lemma}[theorem]{Lemma}{}

\theoremstyle{definition}

\newcommand{\bad}{\mathrm{b}}

\definecolor{rootcolor}{HTML}{ff8c66}
\definecolor{onecolor}{HTML}{055f96}

\title{Erd\H{o}s-S\'os via random cyclic orderings}
\author{Bryce Frederickson\thanks{Department of Mathematics, Emory University, 
Atlanta, GA, USA. Email: {\tt bfrede4@emory.edu}}}
\date{}

\begin{document}

\maketitle

\begin{abstract}
        We present a simplified proof of the Erd\H{o}s--S\'os conjecture on trees in graphs, which is based on its recent resolution by GPT-6 Astra. At the same time, we also give another proof of the corresponding conjecture for antidirected trees in digraphs, posed by Addario-Berry, Havet, Linhares Sales, Reed, and Thomass\'e, and we simplify recent proofs of Kalai's conjecture for tight trees in hypergraphs. Our proof uses the language of random cyclic orderings in order to take advantage of symmetries that were obscured in the original argument, and we hope that the reader will find it intuitive. We also characterize the extremal examples for Kalai's conjecture, and we generalize it to the setting of so-called fully directed $r$-uniform hypergraphs, whose edges are ordered $r$-tuples.
\end{abstract}

\section{Introduction}

In 1959, Erd\H{o}s and Gallai~\cite{gallai1959maximal} proved that for every integer $k \geq 1$, every graph with average degree strictly greater than $k-1$ contains a path with $k$ edges. If we replace ``average degree'' with ``minimum degree'' in this statement, then such a path can be obtained by greedily embedding one vertex at a time. In fact, the same argument shows that every tree with $k$ edges can be obtained as a subgraph of any graph with minimum degree greater than $k-1$. Erd\H{o}s and S\'os conjectured (see~\cite{Erdos1964Extremal}) that the same should hold with average degree as it does for paths. In a shock to many, OpenAI~\cite{astra2026counting} recently announced a short proof of the full Erd\H{o}s--S\'os conjecture by the AI model GPT-6 Astra.

\begin{theorem}[Erd\H{o}s--S\'os conjecture, GPT-6 Astra~\cite{astra2026counting}]\label{conj:main}
    For every integer $k \geq 1$, every graph of average degree more than $k-1$ contains every tree with $k$ edges as a subgraph.
\end{theorem}

Note that this average degree condition cannot be improved; a vertex-disjoint union of cliques of size $k$ excludes every $k$-edge tree and has average degree $k-1$, and any $(k-1)$-regular graph excludes a $k$-edge star.

The Erd\H{o}s--S\'os conjecture has received much attention over the years, and has led to several interesting breakthroughs and perspectives. Some of these resolved \Cref{conj:main} for special classes of trees~\cite{fan2018erd,mclennan2005erdHos, Pokrovskiy2024Hyperstability}, others for special classes of host graphs~\cite{balasubramanian2007constructing,brandt1996erdHos,dobson2002constructing,gorlich2016erdHos,sacle1997erdHos,STEIN2024kalai}, while others proved asymptotic versions with various additional hypotheses~\cite{BPS2021BoundedDegree,DPRSM2026Asymptotic,Pokrovskiy2024Hyperstability,ReedStein2025NearlySpanning,ReedStein2026Dense,rozhon2019local}. In 2015, Ajtai, Koml\'os, Simonovits, and Szemer\'edi~\cite{ajtai2015exact} announced an elaborate proof of the Erd\H{o}s--S\'os conjecture for sufficiently large $k$ using Szemer\'edi's regularity lemma, but this proof was never shared publicly. \Cref{conj:main} has also been generalized in several directions, such as for digraphs~\cite{ABHLRT2013OrientedTrees}, for edge-colored graphs~\cite{CKS2025Rainbow}, and for hypergraphs (see~\cite{frankl1987exact}). See~\cite{Stein2020Survey} for a survey by Stein.

Following the recent proof of the Erd\H{o}s--S\'os conjecture~\cite{astra2026counting}, several papers have come out explaining, simplifying, and extending their method~\cite{debiasio2026kalai,debiasio2026erdos,liu2026antidirected,MubayiVerstraete2026Digraphs,MubayiVerstraete2026kalai,riordan2026shortproof,SantosSteinWilliams2026Butterflies,Wood2026ErdosSos}. Two such extensions established strengthenings of \Cref{conj:main} for digraphs and for hypergraphs, conjectured by Addario-Berry, Havet, Linhares Sales, Reed, and Thomass\'e~\cite{ABHLRT2013OrientedTrees} and by Kalai (see~\cite{frankl1987exact}), respectively. The digraph generalization was proved independently by Riordan and Scott~\cite{riordan2026shortproof}, by Santos, Stein, and Williams~\cite{SantosSteinWilliams2026Butterflies}, and by GPT-6 Astra~\cite{debiasio2026erdos}, while the hypergraph generalization was proved by GPT-6 Astra in two separate instances~\cite{debiasio2026kalai, MubayiVerstraete2026kalai}. This paper likewise aims to advance human understanding of these new proofs of the Erd\H{o}s--S\'os conjecture, its digraph extension from~\cite{ABHLRT2013OrientedTrees}, and Kalai's conjecture for hypergraphs.

\subsection{Antidirected trees in digraphs}

We now introduce some terminology for the purpose of stating these results properly, beginning with the digraph version. We assume that all digraphs are loopless and have no multiple arcs with the same source and sink, but they may have \emph{digons}, which are directed cycles of length $2$. A digraph with no digons is called an \emph{oriented graph}. An oriented graph is called \emph{antidirected} if every vertex is a source or a sink. The conjecture of Addario-Berry, Havet, Linhares Sales, Reed, and Thomass\'e~\cite{ABHLRT2013OrientedTrees} states the following.

\begin{theorem}[Riordan and Scott~\cite{riordan2026shortproof}, Santos, Stein, Williams~\cite{SantosSteinWilliams2026Butterflies}, GPT-6 Astra~\cite{debiasio2026erdos}]\label{thm:main directed}
    For every integer $k \geq 1$, every digraph of average out-degree greater than $k-1$ contains every antidirected tree with~$k$ arcs as a subdigraph.
\end{theorem}

Note that the Erd\H{o}s--S\'os conjecture follows from \Cref{thm:main directed} by replacing each edge in an undirected host graph of average degree $d$ with a digon in order to obtain a digraph with average out-degree $d$. Conversely, as observed by Burr~\cite{Burr1982Antidirected}, the Erd\H{o}s--S\'os conjecture implies a weak form of \Cref{thm:main directed} with the average out-degree assumption increased to $4k$, which follows by considering a largest directed cut in the host digraph. 

When the authors of~\cite{ABHLRT2013OrientedTrees} posed their conjecture, they verified it for trees of diameter at most~$3$. Stein and Z\'arate-Guer\'en~\cite{SteinZarateGueren2024Antidirected} proved \Cref{thm:main directed} asymptotically for trees of small maximum degree when the host graph is dense. Stein and Trujillo-Negrete~\cite{SteinTrujilloNegrete2025Dense} proved \Cref{thm:main directed} exactly when the host digraph forbids certain subdigraphs, and also for arbitrary host graphs when $T$ is an antidirected caterpillar. This latter result for caterpillars was recently extended extended by Santos, Stein, and Williams~\cite{SantosSteinWilliams2026Butterflies} to an asymptotic result for arbitrary antidirected trees of bounded degree in dense host graphs. Additional prior work on this and related problems can be found in~\cite{GaoLo2026Antipaths, GrzesikSkrzypczyk2025Antidirected,    KlimosovaStein2023Antipaths,  KontogeorgiouSantosStein2025Antidirected,  liu2026antidirected, LuChen2025Girth, SantosSteinWilliams2026Butterflies, SteinTrujilloNegrete2025Dense,  SteinZarateGueren2024Antidirected}, and in Stein's surveys~\cite{Stein2020Survey,Stein2024OrientedSurvey}.

\subsection{Tight trees in hypergraphs and Kalai's conjecture}

Next, we discuss the hypergraph version. An $r$-uniform \emph{tight tree} is an $r$-uniform hypergraph $T$ with no isolated vertices whose edges admit a \emph{tight tree ordering}, which is an ordering $(e_1, \ldots, e_k)$ of $E(T)$ with the property that for every $2 \leq j \leq k$, the edge $e_j$ contains exactly one vertex outside $\bigcup_{i < j} e_i$, and the remaining $r-1$ vertices of $e_j$ are all contained in a single edge $e_i$ with $i < j$. Equivalently, $(e_1, \ldots, e_k)$ is a tight tree ordering of $E(T)$ if every edge $e_j$ with $2 \leq j \leq k$ intersects some previous edge $e_i$ with $i < j$ in $r-1$ vertices, and $\left|\bigcup_{j=1}^k e_j \right| = r + k-1$. Note that $2$-uniform tight trees are precisely trees in the graph sense.

\begin{theorem}[Kalai's conjecture, GPT-6 Astra~\cite{debiasio2026kalai,MubayiVerstraete2026kalai}]\label{thm:main hypergraph}
    Let $k \geq 1$ be an integer, and let $H$ be an $n$-vertex, $r$-uniform hypergraph with more than $\frac{k-1}{r}\binom{n}{r-1}$ edges. Then $H$ contains every $r$-uniform tight tree with $k$ edges as a subhypergraph.
\end{theorem}

The $r=2$ case of Kalai's conjecture is precisely the Erd\H{o}s--S\'os conjecture. Kalai's conjecture in itself has a long history of progress prior to its resolution. In 1987, Frankl and F\"uredi~\cite{frankl1987exact} proved that Kalai's conjecture holds when $T$ is \emph{star-shaped}, meaning that there exists an edge of $T$ intersecting all other edges in $r-1$ vertices. Much later, in 2019, F\"uredi, Jiang, Kostochka, Mubayi, Verstra\"ete~\cite{furedi2019hypergraphs} extended this result, in an asypmtotic form, to tight trees $T$ with a \emph{bounded trunk}, meaning that there exists a tight subtree $T'$ of $T$ of bounded size such that every edge of $T \setminus T'$ intersects some edge of $T'$ in $r-1$ vertices. They also verified Kalai's conjecture exactly for all trees with at most four edges, and later~\cite{furedi2020hypergraphs} for sufficiently large $3$-uniform tight trees with trunks of size $2$.  F\"uredi and Jiang~\cite{furedi2015tur} also gave asymptotic and exact results for other families of tight trees. There was also partial progress when $T$ is a tight path~\cite{allen2017tight,furedi2020tight,patkos2012note}, but the best bounds were still asymptotically above the conjectured bound of \Cref{thm:main hypergraph} for $r \geq 3$ except in the dense regime where $H$ is not much larger than $T$. In another direction, Stein~\cite{STEIN2024kalai} proved Kalai's conjecture in the case where the host graph $H$ is $r$-partite. See~\cite{Stein2020Survey} for more details.

For every $k \geq 1$ and $r \geq 2$, \Cref{thm:main hypergraph} is optimal for infinitely many positive integers $n$. To describe the extremal constructions, we give the following definitions.

A \emph{$t$-$(n,r,\lambda)$-design} is a $n$-vertex, $r$-uniform hypergraph with the property that every $t$-element set of vertices is contained in exactly $\lambda$ edges. In the $\lambda = 1$ case, we call a $t$-$(n,r,1)$-design a \emph{Steiner system} with parameters $t,r,n$, denoted $S(t,r,n)$. Keevash~\cite{keevash2014existence} (see also~\cite{delcourt2024refined} and~\cite{glock2023existence} for alternative proofs) proved that for every $t, r, \lambda$, there exist infinitely many values of $n$ for which $t$-$(n,r,\lambda)$-designs exist.

A \emph{sunflower} is a hypergraph with the property that any two edges have the same pairwise intersection. If $T$ is an $r$-uniform tight tree which is also a sunflower, this means that there exists a set $f$ with $r-1$ vertices such that every edge of $T$ contains $f$ and one other vertex. If such a sunflower has $k$ edges, then any $(r-1)$-$(n,r,k-1)$-design excludes $T$ as a subhypergraph and has precisely $\frac{k-1}{r} \binom{n}{r-1}$ edges.

For an arbitrary tight tree $T$ with $k$ edges, it is possible to construct special $(r-1)$-$(n,r,k-1)$-designs which are $T$-free as follows. Simply begin with an $S(r-1,k+r-2,n)$, then replace each $(k+r-2)$-vertex edge $e$ with an $r$-uniform clique $K_{k+r-2}^{(r)}$ on vertex set $e$ to obtain an $r$-uniform hypergraph $H$ which is an $(r-1)$-$(n,r,k-1)$-design. The \emph{tight components} of $H$ (defined at the end of this section) are precisely these cliques, and so any tight tree in~$H$ can have at most $k+r-2$ vertices, and hence at most $k-1$ edges.

We show that these constructions characterize all extremal examples for \Cref{thm:main hypergraph}. In~\cite{MubayiVerstraete2026kalai}, Mubayi and Verstra\"ete stated that GPT-6 Astra claimed to have a proof of this characterization which they had not verified.

\begin{theorem}\label{thm:extremal cases hypergraph}
    Let $T$ be an $r$-uniform tight tree with $k \geq 1$ edges, and let $H$ be an $n$-vertex, $r$-uniform hypergraph with exactly $\frac{k-1}{r} \binom{n}{r-1}$ edges and no copy of $T$. Then $H$ is an $(r-1)$-$(n,r,k-1)$-design.
    
    Moreover, if $T$ is not a sunflower, then $H$ is obtained from an $S(r-1,k+r-2,n)$ by replacing each $(k+r-2)$-element edge by a complete $r$-uniform hypergraph on the same set of vertices.
\end{theorem}

The $r=2$ case characterizing extremal examples for the Erd\H{o}s--S\'os conjecture was done by Riordan and Scott~\cite{riordan2026shortproof}. In this case, these tight examples are precisely the $(k-1)$-regular graphs if $T$ is a star, and the vertex-disjoint unions of cliques of size $k$ for all other trees $T$.

\subsection{Our purpose}

Our main objective in this article is to give a new and simplified presentation of the proof of \Cref{conj:main} given in~\cite{astra2026counting}. We present two variations on this proof; the first is catered to the directed setting to establish \Cref{thm:main directed}, while the second is catered to the hypergraph setting to recover \Cref{thm:main hypergraph}. Our proof of \Cref{thm:main directed} is based entirely on Bloom's exposition~\cite{bloom2026exposition} of the proof of the Erd\H{o}s--S\'os conjecture in~\cite{astra2026counting}, and was obtained independently of the other proofs of \Cref{thm:main directed} given in~\cite{debiasio2026erdos,riordan2026shortproof,SantosSteinWilliams2026Butterflies} (see \Cref{sec:comparison} for more discussion on how these proofs compare). On the other hand, we do not claim \Cref{thm:main hypergraph} to be our own independent result; rather, our purpose in including its proof is to clarify the main ideas from~\cite{debiasio2026kalai} and~\cite{MubayiVerstraete2026kalai} and provide an arguably simpler proof of~\Cref{thm:main hypergraph} in probabilistic language. Our novel contribution in the hypergraph setting is \Cref{thm:extremal cases hypergraph}, which characterizes the extremal examples of Kalai's conjecture. Additionally, in \Cref{sec:concluding remarks}, we show how our proof of \Cref{thm:main hypergraph} can be adapted to give a common generalization of both \Cref{thm:main directed} and \Cref{thm:main hypergraph}.

The two proofs share many of the same ideas, but the particulars of the presentation are somewhat different. For instance, both proofs involve randomly permuting the vertex set of the host digraph or hypergraph, but our proof of \Cref{thm:main directed} considers a uniformly random \emph{cyclic ordering}, while our proof of \Cref{thm:main hypergraph} considers a uniformly random \emph{linear ordering}. Fundamentally, nothing inhibits us from using the other type of ordering in either case, but there are certain aspects of each proof that we find more intuitive and convenient when presented in the chosen way. In particular, we recommend our proof of \Cref{thm:main directed} using cyclic orderings to the reader who simply wants to understand the new proof of Erd\H{o}s--S\'os conjecture from~\cite{astra2026counting} ``at a glance'', and we have provided figures to assist in visualizing the main components of the proof.

\vspace{10pt}
\noindent
\textbf{Notation.} The notation and terminology that we use for digraphs and hypergraphs is standard, unless we specify otherwise. For an $r$-uniform hypergraph $H$ and a subset $S \subseteq V(H)$, we use $\deg_H(S)$ to denote the number of edges of $H$ containing $S$. Let $J(H)$ denote the graph on vertex set $E(H)$, with two edges of $H$ adjacent in $J(H)$ whenever they intersect in precisely $r-1$ vertices. We say that $H$ is \emph{tightly connected} if $H$ has no isolated vertices and $J(H)$ is connected. A connected component in $J(H)$ corresponds to a set $E'$ of edges in $H$, which are supported on some vertex set $V'$ of $H$, and the subhypergraph $H' = (V',E')$ is called a \emph{tight component} of $H$. Note that the tight components of $H$ form an edge-partition of $H$, and any two tight components intersect in at most $r-2$ vertices.

\vspace{10pt}
\noindent
\textbf{Overview.} In \Cref{sec:antidirected}, we prove \Cref{thm:main directed} for antidirected trees in digraphs. In \Cref{sec:kalai}, we give our simplified proof of Kalai's conjecture (\Cref{thm:main hypergraph}). In \Cref{sec:extremal cases}, we prove the characterization of the extremal cases for Kalai's conjecture described in \Cref{thm:extremal cases hypergraph}. Finally, in \Cref{sec:concluding remarks}, we state our common generalization of \Cref{thm:main directed} and \Cref{thm:main hypergraph} and sketch its proof.

\section{\texorpdfstring{Erd\H{o}s--S\'os}{Erdős-Sós} in the directed setting}\label{sec:antidirected}
This section is devoted to the proof of \Cref{thm:main directed}.
\subsection{Notation and setup}
Let $D$ be a fixed $n$-vertex digraph throughout. Let $C_D$ denote the set of cyclic orderings of $V(D)$, which we identify with the set of ordered $n$-tuples of the form $(v_1, \ldots, v_n)$ (indices are modulo $n$) with $V(D) = \{v_1, \ldots, v_n\}$, modulo the relation $\sim$ given by $(v_1, \ldots, v_n) \sim (v_j, v_{j+1}, \ldots, v_{j-1})$ for all $1 \leq j \leq n$. Given a cyclic ordering $\pi \in C_D$ represented by $(v_1, \ldots, v_n)$, the \emph{interval} of $\pi$ from $v_i$ to $v_j$ is the sequence $[v_i, v_j]_\pi := (v_i, v_{i+1}, \ldots, v_j)$ (depicted counterclockwise from $v_i$ to $v_j$ in our figures). The $\sim$-equivalence class of $(v_n, \ldots, v_1)$ is called the \emph{reversal} of $\pi$, which we denote by $\overline \pi$.

A \emph{rooted arc} of $D$ is an ordered pair $(e, u)$, where $e \in E(D)$, and $u \in V(D)$ is a designated \emph{root vertex} incident to $e$ in $D$. If $u$ is the source of $e$, then $(e,u)$ is called a \emph{rooted out-arc}. If $u$ is the sink of $e$, then $(e,u)$ is called a \emph{rooted in-arc}.
Let $T$ be an oriented tree with at least one arc.
If $r \in V(T)$, and $e \in E(T)$ is incident to $r$ in $T$, we call the triple $(T,r,e)$ a \emph{strongly rooted oriented tree} with \emph{root vertex} $r$ and \emph{root arc} $e$.
Given a cyclic ordering $\pi \in C_D$, we say that a rooted arc $(uv, u)$ of $D$ \emph{spans} a copy of $(T,r,e)$ in $\pi$ if there is an embedding of $T$ into~$D$ mapping $V(T)$ into the set of vertices in $[u,v]_\pi$, and mapping $r$ to $u$ and $e$ to~$uv$.
Let $\bad_\pi^{\mathrm{out}} (T,r,e)$ and $\bad_\pi^{\mathrm{in}} (T,r,e)$ denote the number of rooted out-arcs and in-arcs of $D$, respectively, which do \emph{not} span a copy of $(T,r,e)$ in~$\pi$. Similarly, for each vertex $u \in V(D)$, let $\bad_{\pi, u}^{\mathrm{out}} (T, r, e)$ and $\bad_{\pi, u}^{\mathrm{in}} (T, r, e)$ denote the respective numbers of such rooted out-arcs and in-arcs of $D$ with root vertex~$u$.
\begin{figure}[h]
\centering
\begin{tikzpicture}[vertices/.style={draw, fill=black, circle, inner sep=0pt, minimum size = 4pt, outer sep=0pt}, rootvertices/.style={draw, color=rootcolor, fill=rootcolor, star, inner sep=0pt, minimum size = 8pt, outer sep=0pt}]

\tikzset{>=latex}

\pgfmathsetmacro{\Rad}{1.8}
\pgfmathsetmacro{\Num}{15}
\pgfmathsetmacro{\light}{0.7}
\pgfmathsetmacro{\heavy}{2}

\node[rootvertices] (r) at (-5,1) {};
\node at (-5,1.5) {\Large \textcolor{rootcolor}{$r$}};
\node[vertices] (t_1) at (-6,0) {};
\node[vertices] (t_11) at (-6.5,-1) {};
\node[vertices] (t_12) at (-5.5,-1) {};
\node[vertices] (t_3) at (-5.2,0) {};
\node[vertices] (t_2) at (-4,0) {};
\node[vertices] (t_21) at (-3.5,-1) {};

\path [->,draw=black, line width= \light] (r) edge (t_1);
\path [<-,draw=black, line width= \light] (t_1) edge (t_11);
\path [<-,draw=black, line width= \light] (t_1) edge (t_12);
\path [->,draw=black, line width= \light] (r) edge (t_3);
\path [->,fill=rootcolor,draw=rootcolor, line width= \heavy] (r) edge (t_2);
\node at (-4.35, 0.85) {\Large \textcolor{rootcolor}{$e$}};
\path [<-,draw=black, line width= \light] (t_2) edge (t_21);

\node at (-5,-1.5) {\Large $T$};

\draw[teal, line width= 5,opacity=0.2] (0,0) circle (\Rad);

\foreach \i in {1,...,\Num} {
\ifnum\i=1
            \node[rootvertices] (v_\i) at (\i * 360/\Num - 6*360/\Num : \Rad) {};
        \fi
\ifnum\i>1
            \node[vertices] (v_\i) at (\i * 360/\Num - 6*360/\Num : \Rad) {};
        \fi
}
\node at (1 * 360/\Num - 6*360/\Num : \Rad + 0.4) {\Large \textcolor{rootcolor}{$u$}};
\node at (8 * 360/\Num - 6*360/\Num : \Rad + 0.4) {\Large $v$};

\path [->, draw=black, line width= \light, bend left = 30] (v_1) edge (v_5);
\path [<-, draw=black, line width= \light, bend right = 30] (v_5) edge (v_4);
\path [<-, draw=black, line width= \light, bend right = 30] (v_5) edge (v_3);
\path [->, draw=black, line width= \light, bend left = 30] (v_1) edge (v_2);
\path [->, fill=rootcolor, draw=rootcolor, line width= \heavy, bend left = 30] (v_1) edge (v_8);
\node at (-.55,.55) {\Large \textcolor{rootcolor}{$uv$}};
\path [<-, draw=black, line width= \light, bend right = 30] (v_8) edge (v_6);

\node at (0,-2.5) {\large $D$, ordered by $\pi$};
\end{tikzpicture}
\caption{A rooted out-arc $(uv,u)$ in $D$ spanning a copy of the strongly rooted oriented tree $(T,r,e)$ in $\pi$.}\label{fig:spanning arc directed}
\end{figure}
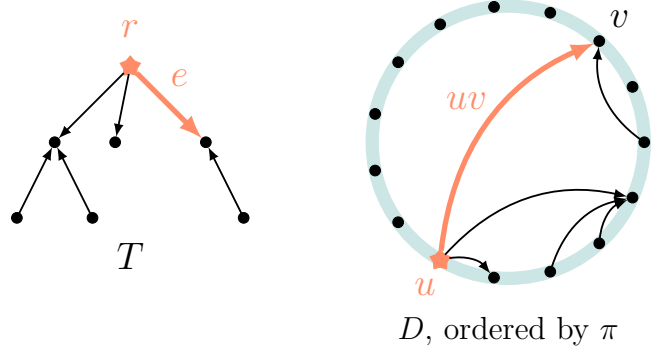

\subsection{Main inductive lemma}
The main lemma that we will show is the following, which we will prove by induction on $e(T)$.

\begin{lemma}\label{lem:main random directed}
    Let $\pi \in C_D$ be a uniformly random cyclic ordering of $V(D)$. Let $(T,r,e)$ be an antidirected strongly rooted oriented tree. If $r$ is a source, then
    \[\mathbb E[\bad_\pi^{\mathrm{out}}(T,r,e)] \leq (e(T) - 1)n.\]
    If $r$ is a sink, then
    \[\mathbb E[\bad_\pi^{\mathrm{in}}(T,r,e)] \leq (e(T) - 1)n.\]
\end{lemma}

Note that \Cref{thm:main directed} immediately follows. 
\begin{proof}[Proof of \Cref{thm:main directed}]
Let $T$ be an antidirected tree with $k \geq 1$ arcs, and let $D$ be an $n$-vertex digraph with average out-degree $d > k-1$. Let $r$ be any source vertex of $T$, and let $e$ be any arc of $T$ incident to $r$. Let $\pi \in C_D$ be chosen uniformly at random. Since $D$ has exactly~$dn$ rooted out-arcs, which is greater than $(k-1)n \geq \mathbb E[\bad_\pi^{\mathrm{out}}(T,r,e)]$ by \Cref{lem:main random directed}, we know that some rooted arc of~$D$ must span a copy of $(T,r,e)$ in some cyclic ordering $\pi$ of $V(D)$.
\end{proof}

The rest of this section is devoted to proving \Cref{lem:main random directed}. The base case is trivial. The inductive step will consist of two steps. First, we show that we can always ``pivot'' the root vertex of our tree $T$ across the root arc for the sake of counting $\mathbb E[\bad_\pi^{\mathrm{out}}(T,r,e)]$ or $\mathbb E[\bad_\pi^{\mathrm{in}}(T,r,e)]$, which allows us to focus on the case in which $r$ is not a leaf of $T$. Second, we obtain a bound on $\mathbb E[\bad_\pi^{\mathrm{out}}(T,r,e)]$ or $\mathbb E[\bad_\pi^{\mathrm{in}}(T,r,e)]$ in terms of corresponding counts on two proper subtrees of $T$. The following lemma accomplishes the first step.

\begin{lemma}\label{lem:pivot}
    Let $\pi \in C_D$ be a uniformly random cyclic ordering of $V(D)$. Then for any antidirected tree $T$ and any arc $e = rr'$ of $T$, we have
    \[\mathbb E[\bad_\pi^{\mathrm{out}}(T,r,e)] = \mathbb E[\bad_\pi^{\mathrm{in}}(T,r',e)].\]
\end{lemma}
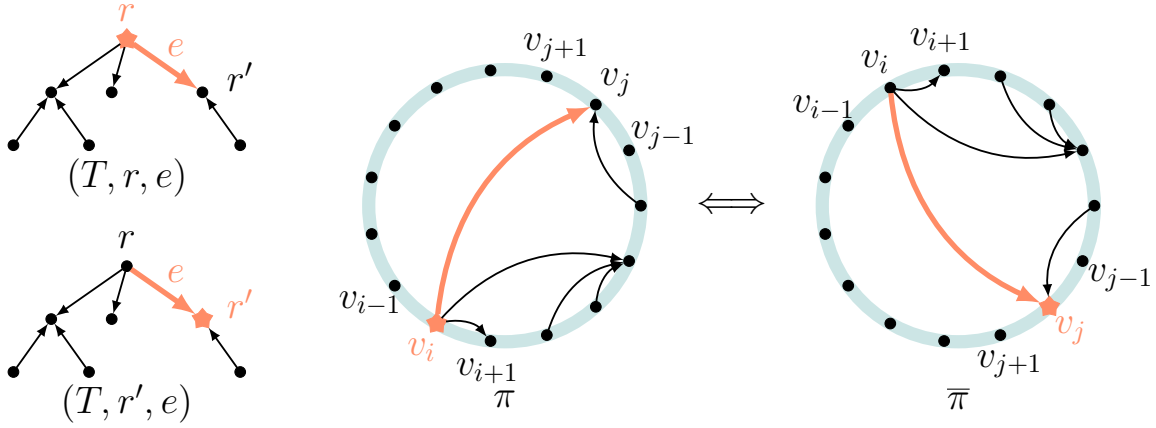
\begin{figure}[h]
\centering
\begin{tikzpicture}[vertices/.style={draw, fill=black, circle, inner sep=0pt, minimum size = 4pt, outer sep=0pt}, rootvertices/.style={draw, color=rootcolor, fill=rootcolor, star, inner sep=0pt, minimum size = 8pt, outer sep=0pt}]

\tikzset{>=latex}

\pgfmathsetmacro{\Rad}{1.8}
\pgfmathsetmacro{\Num}{15}
\pgfmathsetmacro{\light}{0.7}
\pgfmathsetmacro{\heavy}{2}
\pgfmathsetmacro{\treeone}{1.5}
\pgfmathsetmacro{\treetwo}{-1.5}
\pgfmathsetmacro{\treestretch}{0.7}

\node[rootvertices] (r) at ($(-5,1*\treestretch) + (0,\treeone)$) {};
\node at ($(-5,1.5*\treestretch)+ (0,\treeone)$) {\Large \textcolor{rootcolor}{$r$}};
\node[vertices] (t_1) at ($(-6,0*\treestretch)+ (0,\treeone)$) {};
\node[vertices] (t_11) at ($(-6.5,-1*\treestretch)+ (0,\treeone)$) {};
\node[vertices] (t_12) at ($(-5.5,-1*\treestretch)+ (0,\treeone)$) {};
\node[vertices] (t_3) at ($(-5.2,0*\treestretch)+ (0,\treeone)$) {};
\node[vertices] (t_2) at ($(-4,0*\treestretch)+ (0,\treeone)$) {};
\node at ($(-3.5,0.3*\treestretch)+ (0,\treeone)$) {\Large $r'$};
\node[vertices] (t_21) at ($(-3.5,-1*\treestretch)+ (0,\treeone)$) {};

\path [->,draw=black, line width= \light] (r) edge (t_1);
\path [<-,draw=black, line width= \light] (t_1) edge (t_11);
\path [<-,draw=black, line width= \light] (t_1) edge (t_12);
\path [->,draw=black, line width= \light] (r) edge (t_3);
\path [->,fill=rootcolor,draw=rootcolor, line width= \heavy] (r) edge (t_2);
\node at ($(-4.35, 0.85*\treestretch) + (0,\treeone)$) {\Large \textcolor{rootcolor}{$e$}};
\path [<-,draw=black, line width= \light] (t_2) edge (t_21);

\node at ($(-5,-1.6*\treestretch)+ (0,\treeone)$) {\Large $(T,r,e)$};

\node[vertices] (r) at ($(-5,1*\treestretch) + (0,\treetwo)$) {};
\node at ($(-5,1.5*\treestretch)+ (0,\treetwo)$) {\Large $r$};
\node[vertices] (t_1) at ($(-6,0*\treestretch)+ (0,\treetwo)$) {};
\node[vertices] (t_11) at ($(-6.5,-1*\treestretch)+ (0,\treetwo)$) {};
\node[vertices] (t_12) at ($(-5.5,-1*\treestretch)+ (0,\treetwo)$) {};
\node[vertices] (t_3) at ($(-5.2,0*\treestretch)+ (0,\treetwo)$) {};
\node[rootvertices] (t_2) at ($(-4,0*\treestretch)+ (0,\treetwo)$) {};
\node at ($(-3.5,0.3*\treestretch)+ (0,\treetwo)$) {\Large \textcolor{rootcolor}{$r'$}};
\node[vertices] (t_21) at ($(-3.5,-1*\treestretch)+ (0,\treetwo)$) {};

\path [->,draw=black, line width= \light] (r) edge (t_1);
\path [<-,draw=black, line width= \light] (t_1) edge (t_11);
\path [<-,draw=black, line width= \light] (t_1) edge (t_12);
\path [->,draw=black, line width= \light] (r) edge (t_3);
\path [->,fill=rootcolor,draw=rootcolor, line width= \heavy] (r) edge (t_2);
\node at ($(-4.35, 0.85*\treestretch) + (0,\treetwo)$) {\Large \textcolor{rootcolor}{$e$}};
\path [<-,draw=black, line width= \light] (t_2) edge (t_21);

\node at ($(-5,-1.6*\treestretch)+ (0,\treetwo)$) {\Large $(T,r',e)$};

\draw[teal, line width= 5,opacity=0.2] (0,0) circle (\Rad);

\foreach \i in {1,...,\Num} {
\ifnum\i=1
            \node[rootvertices] (v_\i) at (\i * 360/\Num - 6*360/\Num : \Rad) {};
        \fi
\ifnum\i>1
            \node[vertices] (v_\i) at (\i * 360/\Num - 6*360/\Num : \Rad) {};
        \fi
}
\node at (1 * 360/\Num - 6*360/\Num : \Rad + 0.4) {\Large \textcolor{rootcolor}{$v_i$}};
\node at (2 * 360/\Num - 6*360/\Num : \Rad + 0.4) {\Large $v_{i+1}$};
\node at (\Num * 360/\Num - 6*360/\Num : \Rad + 0.4) {\Large $v_{i-1}$};
\node at (8 * 360/\Num - 6*360/\Num : \Rad + 0.4) {\Large $v_j$};
\node at (9 * 360/\Num - 6*360/\Num : \Rad + 0.4) {\Large $v_{j+1}$};
\node at (7 * 360/\Num - 6*360/\Num : \Rad + 0.5) {\Large $v_{j-1}$};

\path [->, draw=black, line width= \light, bend left = 30] (v_1) edge (v_5);
\path [<-, draw=black, line width= \light, bend right = 30] (v_5) edge (v_4);
\path [<-, draw=black, line width= \light, bend right = 30] (v_5) edge (v_3);
\path [->, draw=black, line width= \light, bend left = 30] (v_1) edge (v_2);
\path [->, fill=rootcolor, draw=rootcolor, line width= \heavy, bend left = 30] (v_1) edge (v_8);
\path [<-, draw=black, line width= \light, bend right = 30] (v_8) edge (v_6);

\node at (0,-2.55) {\Large $\pi$};

\node at (3,0) {\Large $\Longleftrightarrow$};

\draw[teal, line width= 5,opacity=0.2] (6,0) circle (\Rad);

\foreach \i in {1,...,\Num} {
\ifnum\i=8
            \node[rootvertices] (v'_\i) at ($(-\i * 360/\Num + 6*360/\Num : \Rad) + (6,0)$) {};
        \fi
\ifnum\i<8
            \node[vertices] (v'_\i) at ($(-\i * 360/\Num + 6*360/\Num : \Rad) + (6,0)$) {};
        \fi
\ifnum\i>8
            \node[vertices] (v'_\i) at ($(-\i * 360/\Num + 6*360/\Num : \Rad) + (6,0)$) {};
        \fi

}
\node at ($(-1 * 360/\Num + 6*360/\Num : \Rad + 0.4) + (6,0)$) {\Large $v_i$};
\node at ($(-2 * 360/\Num + 6*360/\Num : \Rad + 0.4) + (6,0)$) {\Large $v_{i+1}$};
\node at ($(-\Num * 360/\Num + 6*360/\Num : \Rad + 0.4) + (6,0)$) {\Large $v_{i-1}$};
\node at ($(-8 * 360/\Num + 6*360/\Num : \Rad + 0.42) + (6,0)$) {\Large \textcolor{rootcolor}{$v_j$}};
\node at ($(-7 * 360/\Num + 6*360/\Num : \Rad + 0.55) + (6,0)$) {\Large $v_{j-1}$};
\node at ($(-9 * 360/\Num + 6*360/\Num : \Rad + 0.4) + (6,0)$) {\Large $v_{j+1}$};

\path [->, draw=black, line width= \light, bend right = 30] (v'_1) edge (v'_5);
\path [<-, draw=black, line width= \light, bend left = 30] (v'_5) edge (v'_4);
\path [<-, draw=black, line width= \light, bend left = 30] (v'_5) edge (v'_3);
\path [->, draw=black, line width= \light, bend right = 30] (v'_1) edge (v'_2);
\path [->,fill=rootcolor, draw=rootcolor, line width= \heavy, bend right = 30] (v'_1) edge (v'_8);
\path [<-,draw=black, line width= \light, bend left = 30] (v'_8) edge (v'_6);

\node at (6,-2.55) {\Large $\overline \pi$};
\end{tikzpicture}
\caption{The rooted out-arc $(v_iv_j,v_i)$ spans a copy of $(T,r,e)$ in $\pi$ if and only if the rooted in-arc $(v_iv_j,v_j)$ spans a copy of $(T,r',e)$ in $\overline \pi$.}\label{fig:pivot directed}
\end{figure}
\begin{proof}
    Note that a rooted out-arc $(uv,u)$ spans a copy of $(T,r,e)$ in $\pi$ if and only if the rooted in-arc $(uv,v)$ spans a copy of $(T,r',e)$ in $\overline \pi$ (see \Cref{fig:pivot directed}), so $\bad_\pi^{\mathrm{out}}(T,r,e) = \bad_{\overline \pi}^{\mathrm{in}}(T,r',e)$. Now because the reversal map $\pi \mapsto \overline \pi$ is an involution on $C_D$, we have that $\overline \pi$ has the same (uniform) distribution as $\pi$, so
    \[\mathbb E[\bad_\pi^{\mathrm{out}}(T,r,e)] = \mathbb E[\bad_{\overline \pi}^{\mathrm{in}}(T,r',e)] = \mathbb E[\bad_{\pi}^{\mathrm{in}}(T,r',e)]. \qedhere\]
\end{proof}

The next lemma handles the second step.

\begin{lemma}\label{lem:gluing random directed}
Let $\pi \in \bad_n$ be a uniformly random cyclic ordering of $V(D)$. Let $T_1$ and $T_2$ be two oriented trees, each with at least one arc and sharing a common vertex $r$, such that $V(T_1) \cap V(T_2) = \{r\}$, and let $T = T_1 \cup T_2$. For each $i \in \{1,2\}$, let $e_i$ be an arc incident to $r$ in $T_i$.
If $r$ is the source of both arcs $e_1$ and~$e_2$, then
\[\mathbb E[\bad_\pi^{\mathrm{out}}(T,r,e_2)] \leq \mathbb E[\bad_\pi^{\mathrm{out}}(T_1, r, e_1)] + \mathbb E[\bad_\pi^{\mathrm{out}}(T_2,r, e_2)] + n.\]
If $r$ is the sink of both arcs $e_1$ and $e_2$, then
\[\mathbb E[\bad_\pi^{\mathrm{in}}(T,r,e_2)] \leq \mathbb E[\bad_\pi^{\mathrm{in}}(T_1, r, e_1)] + \mathbb E[\bad_\pi^{\mathrm{in}}(T_2,r, e_2)] + n.\]
\end{lemma}
\begin{proof}
   We will assume that $r$ is a source of both $e_1$ and $e_2$. The sink case is completely analogous. Fix a vertex $u \in V(D)$. We define a bijection $\varphi_u : C_D \to C_D$ as follows. Let $\pi \in C_D$, and let $(u, v_2, \ldots, v_n)$ be the representation of $\pi$ whose first component is $u$. Let $j = j_{\pi,u}$ be the minimum index with $2 \leq j \leq n$ such that $(uv_j, u)$ is a rooted out-arc in $D$ that spans a copy of $(T_1, r, e_1)$ in $\pi$, if such an index exists. Otherwise, define $j_{\pi,u} := n$. 
   We then define $\varphi_u(\pi)$ to be the $\sim$-equivalence class of $(v_j, \ldots, v_2, u, v_{j+1}, \ldots, v_n)$.
To see that $\varphi_u$ is bijective, simply note that the composition of $\varphi_u$ with the reversal map, namely $\overline \varphi_u(\pi) := \overline{\varphi_u(\pi)}$, preserves the interval $[u, v_{j_{\pi, u}}]_{\pi}$ for every $\pi \in C_D$, and so $j_{\overline \varphi_u(\pi),u} = j_{\pi,u}$. Consequently, $\overline \varphi_u$ is an involution $C_D \to C_D$, which means that $\pi \mapsto \overline \varphi_u(\overline \pi)$ is the inverse of the function of $\varphi_u$.

\begin{figure}[h]
\centering
\begin{tikzpicture}[vertices/.style={draw, fill=black, circle, inner sep=0pt, minimum size = 4pt, outer sep=0pt},
rootvertices/.style={draw, color=rootcolor, fill=rootcolor, star, inner sep=0pt, minimum size = 8pt, outer sep=0pt}, onevertices/.style={draw, color=onecolor, fill=onecolor, circle, inner sep=0pt, minimum size = 4pt, outer sep=0pt}, twovertices/.style={draw, color=rootcolor, fill=rootcolor, circle, inner sep=0pt, minimum size = 4pt, outer sep=0pt}]

\tikzset{>=latex}

\pgfmathsetmacro{\Rad}{1.8}
\pgfmathsetmacro{\Num}{15}
\pgfmathsetmacro{\light}{0.7}
\pgfmathsetmacro{\heavy}{2}

\node[rootvertices] (r) at (-5,1) {};
\node at (-5,1.5) {\Large \textcolor{rootcolor}{$r$}};
\node[onevertices] (t_1) at (-6,0) {};
\node at (-6.5,0.3) {\textcolor{onecolor}{\Large $T_1$}};
\node[onevertices] (t_11) at (-6.5,-1) {};
\node[onevertices] (t_12) at (-5.5,-1) {};
\node[onevertices] (t_3) at (-5.2,0) {};
\node[twovertices] (t_2) at (-4,0) {};
\node at (-3.5,0.3) {\textcolor{rootcolor}{\Large $T_2$}};
\node[twovertices] (t_21) at (-3.5,-1) {};

\path [->,draw=onecolor,fill=onecolor, line width= \heavy] (r) edge (t_1);
\node at (-5.75, 0.75) {\Large \textcolor{onecolor}{$e_1$}};
\path [<-,draw=onecolor,fill=onecolor, line width= \light] (t_1) edge (t_11);
\path [<-,draw=onecolor,fill=onecolor, line width= \light] (t_1) edge (t_12);
\path [->,draw=onecolor,fill=onecolor, line width= \light] (r) edge (t_3);
\path [->,fill=rootcolor,draw=rootcolor, line width= \heavy] (r) edge (t_2);
\node at (-4.25, 0.75) {\Large \textcolor{rootcolor}{$e_2$}};
\path [<-,draw=rootcolor,fill=rootcolor, line width= \light] (t_2) edge (t_21);

\node at (-5,-1.5) {\Large $T$};

\draw[teal, line width= 5,opacity=0.2] (0,0) circle (\Rad);

\foreach \i in {1,...,\Num} {
\ifnum\i=5
            \node[rootvertices] (v_\i) at (\i * 360/\Num - 6*360/\Num : \Rad) {};
        \fi
\ifnum\i<5
            \node[vertices] (v_\i) at (\i * 360/\Num - 6*360/\Num : \Rad) {};
        \fi
\ifnum\i>5
            \node[vertices] (v_\i) at (\i * 360/\Num - 6*360/\Num : \Rad) {};
        \fi
}
\node at (1 * 360/\Num - 6*360/\Num : \Rad + 0.4) {\Large $v_j$};
\node at (8 * 360/\Num - 6*360/\Num : \Rad + 0.4) {\Large $v_i$};
\node at (5 * 360/\Num - 6*360/\Num : \Rad + 0.4) {\Large \textcolor{rootcolor}{$u$}};
\node at (4 * 360/\Num - 6*360/\Num : \Rad + 0.4) {\Large $v_2$};
\node at (6 * 360/\Num - 6*360/\Num : \Rad + 0.6) {\Large $v_{j+1}$};
\node at (\Num * 360/\Num - 6*360/\Num : \Rad + 0.6) {\Large $v_n$};

\path [->, draw=onecolor,fill=onecolor, line width= \heavy, bend right = 30] (v_5) edge (v_1);
\path [<-, draw=onecolor, fill=onecolor, line width= \light, bend left = 30] (v_1) edge (v_2);
\path [<-, draw=onecolor,fill=onecolor, line width= \light, bend left = 30] (v_1) edge (v_3);
\path [->, draw=onecolor,fill=onecolor, line width= \light, bend right = 30] (v_5) edge (v_4);
\path [->, fill=rootcolor, draw=rootcolor, line width= \heavy, bend left = 30] (v_5) edge (v_8);
\path [<-, draw=rootcolor,fill=rootcolor, line width= \light, bend right = 30] (v_8) edge (v_6);

\node (sw_u) at (5 * 360/\Num - 6*360/\Num : \Rad + .75) {};
\node (sw_vj) at (1 * 360/\Num - 6*360/\Num : \Rad + .75) {};
\path [<->, draw=black, line width = \light, bend left = 40] (sw_u) edge (sw_vj);

\node at (0,-3) {\Large $\varphi_u(\pi)$};

\node at (3.5,0) {\Large $\Rightarrow$};

\draw[teal, line width= 5,opacity=0.2] (6,0) circle (\Rad);

\foreach \i in {1,...,\Num} {
\ifnum\i=1
            \node[rootvertices] (v'_\i) at ($(\i * 360/\Num - 6*360/\Num : \Rad) + (6,0)$) {};
        \fi
\ifnum\i>1
            \node[vertices] (v'_\i) at ($(\i * 360/\Num - 6*360/\Num : \Rad) + (6,0)$) {};
        \fi

}
\node at ($(1 * 360/\Num - 6*360/\Num : \Rad + 0.4) + (6,0)$) {\Large \textcolor{rootcolor}{$u$}};
\node at ($(8 * 360/\Num - 6*360/\Num : \Rad + 0.4) + (6,0)$) {\Large $v_i$};
\node at ($(2 * 360/\Num - 6*360/\Num : \Rad + 0.4) + (6,0)$) {\Large $v_2$};
\node at ($(5 * 360/\Num - 6*360/\Num : \Rad + 0.4) + (6,0)$) {\Large $v_j$};
\node at ($(6 * 360/\Num - 6*360/\Num : \Rad + 0.6)+(6,0)$) {\Large $v_{j+1}$};
\node at ($(\Num * 360/\Num - 6*360/\Num : \Rad + 0.4) + (6,0)$) {\Large $v_n$};

\path [->, draw=onecolor,fill=onecolor, line width= \heavy, bend left = 30] (v'_1) edge (v'_5);
\path [<-, draw=onecolor,fill=onecolor, line width= \light, bend right = 30] (v'_5) edge (v'_4);
\path [<-, draw=onecolor,fill=onecolor, line width= \light, bend right = 30] (v'_5) edge (v'_3);
\path [->, draw=onecolor,fill=onecolor, line width= \light, bend left = 30] (v'_1) edge (v'_2);
\path [->,fill=rootcolor, draw=rootcolor, line width= \heavy, bend left = 30] (v'_1) edge (v'_8);
\path [<-,draw=rootcolor,fill=rootcolor, line width= \light, bend right = 30] (v'_8) edge (v'_6);

\node (sw_u') at ($(5 * 360/\Num - 6*360/\Num : \Rad + .75) + (6,0)$) {};
\node (sw_vj') at ($(1 * 360/\Num - 6*360/\Num : \Rad + .75) + (6,0)$) {};
\path [<->, draw=black, line width = \light, bend left = 40] (sw_u') edge (sw_vj');

\node at (6,-3) {\Large $\pi$};
\end{tikzpicture}
\caption{Every rooted out-arc of $D$ of the form $(uv_i, u)$ with $i > j = j_{\pi, u}$ which spans a copy of $(T_2, r, e_2)$ in $\varphi_u(\pi)$ must also span a copy of $(T, r, e_2)$ in $\pi$.}\label{fig:gluing directed}
\end{figure}
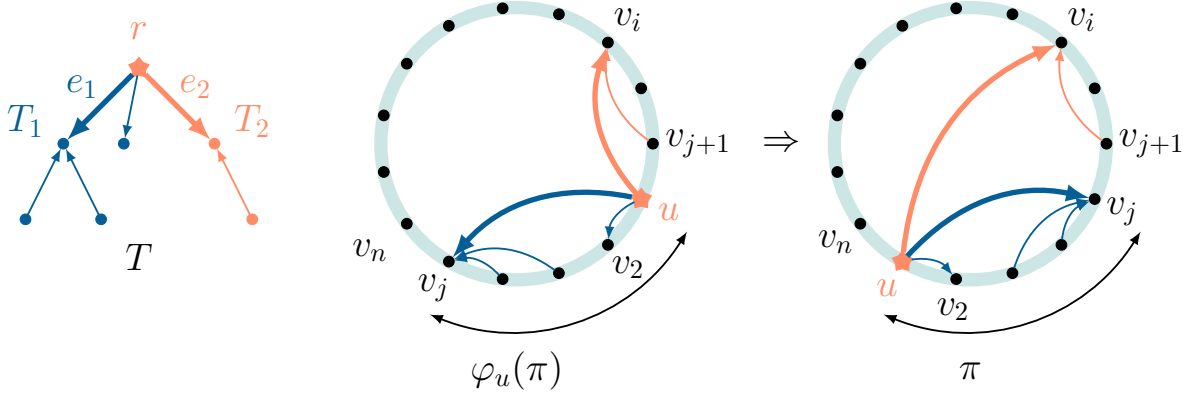

Now consider a fixed cyclic ordering $\pi$ represented by $(u, v_2, \ldots, v_n)$, and let $j = j_{\pi, u}$. We wish to bound $\bad_{\pi, u}^{\mathrm{out}}(T, r, e_2)$ in terms of $\bad_{\pi, u}^{\mathrm{out}}(T_1, r, e_1)$ and $\bad_{\varphi_u(\pi), u}^{\mathrm{out}}(T_2, r, e_2)$. The key observation is that whenever $j < n$, we know that $(uv_j, u)$ spans a copy of $(T_1, r, e_1)$ in $\pi$, so every rooted out-arc of $D$ of the form $(uv_i, u)$ with $i > j$ which spans a copy of $(T_2, r, e_2)$ in $\varphi_u(\pi)$ must also span a copy of $(T, r, e_2)$ in~$\pi$ (see \Cref{fig:gluing directed}). Thus the number of rooted arcs of the form $(uv_i, u)$ with $i > j$ which \emph{do not} span a copy of $(T, r, e_2)$ in $\pi$ is at most $\bad_{\varphi_u(\pi), u}^{\mathrm{out}}(T_2, r, e_2)$ (and this holds even if $j = n$). On the other hand, by the definition of $j = j_{\pi,u}$, every rooted out-arc of $D$ of the form $(uv_i, u)$ with $i < j$ does not span a copy of $(T_1, r, e_1)$ in $\pi$, so the total number of such rooted out-arcs of $D$ is at most $\bad_{\pi,u}^{\mathrm{out}}(T_1, r, e_1)$. The only other index that can contribute to $\bad_{\pi,u}^{\mathrm{out}}(T,r,e_2)$ is $j$ itself, so we obtain
\[\bad_{\pi, u}^{\mathrm{out}}(T, r, e_2) \leq \bad_{\pi, u}^{\mathrm{out}}(T_1, r, e_1) + \bad_{\varphi_u(\pi), u}^{\mathrm{out}}(T_2, r, e_2) + 1.\]

Now let $\pi$ be uniformly random in $C_D$ once again. Since $\varphi_u$ is a bijection $C_D \to C_D$, we have that $\varphi_u(\pi)$ has the same distribution as $\pi$, which means that
\[\mathbb E[\bad_{\pi, u}^{\mathrm{out}}(T, r, e_2)] \leq \mathbb E[\bad_{\pi, u}^{\mathrm{out}}(T_1, r, e_1)] + \mathbb E[\bad_{\pi, u}^{\mathrm{out}}(T_2, r, e_2)] + 1\]
by the linearity of expectation.
Summing over all $u \in V(D)$ now gives
\[\mathbb E[\bad_{\pi}^{\mathrm{out}}(T, r, e_2)] \leq \mathbb E[\bad_{\pi}^{\mathrm{out}}(T_1, r, e_1)] + \mathbb E[\bad_{\pi}^{\mathrm{out}}(T_2, r, e_2)] + n,\]
as desired.
\end{proof}

We are now ready to finish the proof of \Cref{thm:main directed} by proving \Cref{lem:main random directed}.

\begin{proof}[Proof of \Cref{lem:main random directed}]
We prove this by induction on $e(T)$. If $e(T) = 1$ and $r$ is a source in $T$, then every rooted out-arc of~$D$ spans a copy of $(T,r,e)$ in~$\pi$, so $\bad_\pi^{\mathrm{out}}(T,r,e) \equiv 0$. Similarly, if $e(T) = 1$ and $r$ is a sink in $T$, then $\bad_\pi^{\mathrm{in}}(T,r,e) \equiv 0$.

Now assume that $e(T) \geq 2$. Let $r' \in V(T)$ such that $e$ is incident to $r$ and $r'$. Note that at $r$ and $r'$ cannot both be leaves in $T$, and exactly one of them is a source in $T$ and the other is a sink since we assumed that~$T$ was antidirected. Therefore, by replacing $(T,r,e)$ with $(T,r',e)$ if necessary, we can assume by \Cref{lem:pivot} that~$r$ is not a leaf in $T$. We will assume that $r$ is a source, where the sink case is completely analogous. Let~$e_1$ and $e_2$ be two distinct arcs of $T$ incident to $r$, with $e_2 = e$, and let $T_1$ and $T_2$ be arc-disjoint subtrees of $T$ such that $V(T_1) \cap V(T_2) = \{r\}$, $E(T_1) \cup E(T_2) = E(T)$, and $e_i \in E(T_i)$ for each $i \in \{1,2\}$. Now by \Cref{lem:gluing random directed} and the inductive hypothesis, we have
\begin{align*}
    \mathbb E[\bad_\pi^{\mathrm{out}} (T,r,e)] &\leq \mathbb E[\bad_\pi^{\mathrm{out}}(T_1, r, e_1)] + \mathbb E[\bad_\pi^{\mathrm{out}}(T_2,r, e_2)] + n \\
    &\leq (e(T_1) - 1)n + (e(T_2) - 1)n + n \\
    &= (e(T) - 1)n,
\end{align*}
and the proof is complete.
\end{proof}

\subsection{Comparison to other recent proofs}\label{sec:comparison}
We wrote our proof based on the proof exposition of \Cref{conj:main} posted by Thomas Bloom~\cite{bloom2026exposition}, and we became aware of additional proofs of \Cref{conj:main} and \Cref{thm:main directed} in~\cite{debiasio2026erdos,riordan2026shortproof,SantosSteinWilliams2026Butterflies} only after our proof was fully written.

Our proof is most similar to that of Riordan and Scott~\cite{riordan2026shortproof}, which involves taking random \emph{linear} orderings of the vertex set of the host graph. It could just as well have been phrased in terms of cyclic orderings instead, but this is a matter of taste. Aside from such cosmetic differences in presentation, the more fundamental distinction is that, in essence, the only strongly rooted trees they consider are those in which the root arc incident to a leaf. This amounts to them having a slightly weaker inductive statement than the one we have, which requires a very similar, but non-identical, treatment in the inductive step.

The proofs of Santos, Stein, and Williams~\cite{SantosSteinWilliams2026Butterflies} and GPT-6 Astra~\cite{debiasio2026erdos} follow very closely to Thomas Bloom's exposition (adapted for digraphs), and they are based on counting so-called \emph{marked pairs}, which take the form $(\pi, j)$, where $\pi = (v_1, \ldots, v_n)$ is a permutation of $V(G)$, and $v_j$ is a neighbor of $v_1$ in the graph (or digraph). The central arguments of those proofs involve establishing injections between certain subsets of these marked pairs, which are reminiscent of the bijections $C_G \to C_G$ we use (namely the reversal map and $\varphi_u$). The particular injections that are chosen do not all correspond naturally to bijections in the permutation space (or cyclic ordering space), and this makes their exact arguments difficult to phrase nicely in probabilistic terms.

\section{Kalai's conjecture}\label{sec:kalai}
In this section, we give our simplified proof of \Cref{thm:main hypergraph}, based on the prior proofs in~\cite{debiasio2026kalai,MubayiVerstraete2026kalai}.

\subsection{Notation and setup}\label{sec:notation}
Let $H$ be a fixed $n$-vertex, $r$-uniform hypergraph throughout. Let $\mathcal S_H$ denote the set of (linear) orderings $(v_1, \ldots, v_n)$ of $V(H)$. Given $\pi = (v_1, \ldots, v_n) \in \mathcal S_H$ and $j \in [n]$, let $V_j(\pi) := \{v_1, \ldots, v_j\}$ denote the set consisting of the first $j$ elements of $\pi$. We will also consider the left action of permutations on ordered tuples. For $t \geq 1$, a permutation $\sigma : [t] \to [t]$, and an ordered $t$-tuple $(v_1, \ldots, v_t)$, let 
\[\sigma(v_1, \ldots, v_t) := (v_{\sigma^{-1}(1)}, \ldots, v_{\sigma^{-1}(t)})\] denote the ordered $t$-tuple $(v_1', \ldots, v_t')$ obtained from $(v_1, \ldots, v_t)$ by permuting coordinates by $\sigma$, so that $v'_{\sigma(i)} = v_i$ for all $i \in [t]$.

A \emph{strongly rooted tight tree} is a pair $(T, \vec e)$, where $T$ is an $r$-uniform tight tree, and $\vec e = (u_1, \ldots, u_r)$ is an \emph{ordered edge} of $T$, meaning that $\{u_1, \ldots, u_r\} \in E(T)$. Given an ordering $\pi = (v_1, \ldots, v_n) \in \mathcal S_H$, an edge $e \in E(H)$ is said to be \emph{$\pi$-marked} if it is of the form $e = \{v_1, \ldots, v_{r-1}, v_j\}$ for some $r \leq j \leq n$. We say that a $\pi$-marked edge $e = \{v_1, \ldots, v_{r-1}, v_j\} \in E(H)$ \emph{spans} a copy of $(T, \vec e)$ in $\pi$ if there is an embedding of $T$ into $H[V_j(\pi)]$ mapping~$u_i$ to~$v_i$ for all $1 \leq i \leq r-1$ and mapping $u_r$ to $v_j$. On the other hand, we say that $e \in E(H)$ is \emph{$\pi$-bad} for $(T,\vec e)$ if it is $\pi$-marked but does not span a copy of $(T, \vec e)$ in $\pi$. Let $\bad_\pi(T, \vec e)$ denote the number of $\pi$-marked edges of $H$ which are $\pi$-bad for $(T, \vec e)$.\footnote{Recall that in \Cref{sec:antidirected}, when $\pi$ was a cylic ordering, the parameter $\bad_\pi$ counted all $\pi$-bad edges. In contrast, now that $\pi$ is a linear ordering, $\bad_\pi$ only counts $\pi$-bad edges which are $\pi$-marked, so there is a factor of $n$ discrepancy between these two perspectives.}

\subsection{Main inductive lemma}
The main lemma that we will show is the following, which we will prove by induction on $e(T)$.
\begin{lemma}\label{lem:main hypergraph}
    Let $\pi \in \mathcal S_H$ be a uniformly random ordering of $V(H)$. Then for any strongly rooted tree $(T, \vec e)$, we have
    \[\mathbb E[\bad_\pi(T, \vec e)] \leq e(T) - 1.\]
\end{lemma}
We first note that \Cref{thm:main hypergraph} follows immediately from \Cref{lem:main hypergraph}.
\begin{proof}[Proof of \Cref{thm:main hypergraph}]
Let $T$ be an $r$-uniform tight tree with $k$ edges, and let $\vec e$ be an ordered edge of $T$. Then the expected number of $\pi$-marked edges of $H$ is precisely 
\[\frac{1}{\binom{n}{r-1}}\sum_{S \in \binom{V(H)}{r-1}} \deg_H(S) = \frac{re(H)}{\binom{n}{r-1}} > k - 1 \geq \mathbb E[\bad_\pi(T, \vec e)],\]
by \Cref{lem:main hypergraph}, so with positive probability, there exists at least one $\pi$-marked edge of $H$ spanning a copy of $(T, \vec e)$ in $\pi$. Such an edge witnesses an embedding of $T$ into $H$.
\end{proof}

The rest of this section is devoted to proving \Cref{lem:main hypergraph}.

\subsection{Invariance under root edge reordering}

We first record the observation from~\cite{debiasio2026kalai} and~\cite{MubayiVerstraete2026kalai} that when we take $\pi \in \mathcal S_H$ uniformly at random, we have for any strongly rooted tight tree $(T, \vec e)$ that the expectation of $\bad_\pi(T, \vec e)$ is independent of the ordering of $\vec e$.

\begin{lemma}\label{lem:reordering lemma}
    Let $\pi \in \mathcal S_H$ be a uniformly random ordering of $V(H)$. Let $(T, \vec e)$ be a strongly-rooted tight tree, and let $\sigma : [r] \to [r]$ be a permutation. Then
    \[\mathbb E[\bad_\pi(T,\vec e)] = \mathbb E[\bad_\pi(T, \sigma(\vec e))].\]
\end{lemma}
\begin{proof}
    Fix $r \leq j \leq n$, and consider the bijection $\mathcal S_H \to \mathcal S_H,\: \pi \mapsto \tilde \sigma_j(\pi)$ obtained by permuting coordinates by $\sigma$, with $j$ playing the role of $r$, and leaving all other coordinates fixed. More precisely, let $\tilde \sigma : [n] \to [n]$ be the permutation of $[n]$ given by 
    \[\tilde \sigma(i) = \left\{\begin{array}{c l}
        \sigma(i) &  \text{if $i \in [r]$,}\\
        i & \text{otherwise,}
    \end{array}\right.\]
     and define $\tilde \sigma_j := (r \: j) \tilde \sigma (r \: j)$.
    Then, for $\pi \in \mathcal S_H$, recall that $\tilde \sigma_j(\pi)$ is the ordering of $V(H)$ obtained from $\pi$ by permuting coordinates according to $\tilde \sigma_j$. Then, writing $\tilde \sigma_j(\pi) = (v'_1, \ldots, v'_n)$, we have $V_j(\pi) = V_j(\tilde \sigma_j(\pi))$ and $(v'_1, \ldots, v'_{r-1}, v'_j) = \sigma(v_1, \ldots, v_{r-1}, v_j)$.
    This means that any embedding of $T$ into $H[V_j(\pi)] = H[V_j(\tilde \sigma_j(\pi))]$ mapping $\vec e$ to $(v_1, \ldots, v_{r-1}, v_j)$ also maps $\sigma(\vec e)$ to $(v'_1, \ldots, v'_{r-1}, v'_j)$, and vice versa. 
    In other words, $\{v_1, \ldots, v_{r-1}, v_j\}$ is an edge of $H$ which is $\pi$-bad for $(T, \vec e)$ if and only if $\{v'_1, \ldots, v'_{r-1}, v'_j\}$ is an edge of $H$ which is $\tilde \sigma_j(\pi)$-bad for $(T, \sigma(\vec e))$.
    Therefore, since $\pi \mapsto \tilde \sigma_j(\pi)$ is a bijection $\mathcal S_H \to \mathcal S_H$, we see that the probability that $\{v_1, \ldots, v_{r-1}, v_j\}$ is an edge of $H$ which is $\pi$-bad for $(T, \vec e)$ is the same as the probability that $\{v_1, \ldots, v_{r-1}, v_j\}$ is an edge of $H$ which is $\pi$-bad for $(T, \sigma(\vec e))$. Summing these probabilities over all $r \leq j \leq n$ then gives the desired equality of expectations.
\end{proof}

\subsection{The inductive step}
Recall that we want to give an upper bound on the expected number $\bad_\pi(T, \vec e)$ of $\pi$-bad edges for a strongly rooted tight tree $(T, \vec e)$ when we consider a uniformly random ordering $\pi \in \mathcal S_H$ of the vertex set of $H$. The bound that we give (see \Cref{lem:gluing random hypergraph}) will be in terms of the expectations of $\bad_\pi(T_1, \vec e_1)$ and $\bad_\pi(T_2, \vec e_2)$ for two proper subtrees $T_1$ and $T_2$ of $T$ which intersect in $r-1$ vertices from $\vec e$, chosen as in the following tight tree decomposition lemma. This differs from the decomposition lemmas~\cite[Lemma 5.2]{debiasio2026kalai} and~\cite[Lemma 4.1]{MubayiVerstraete2026kalai} used in earlier proofs in that the two trees $T_1$ and $T_2$ intersect in $r-1$ vertices instead of in an edge, and this difference allows us to avoid casework when proving \Cref{lem:main hypergraph}. In our application, $e_2$ will be the underlying edge of $T$ ordered by $\vec e$.

\begin{lemma}\label{lem:tight tree splitting}
    Let $T$ be an $r$-uniform tight tree with $k \geq 2$ edges. Then for any edge $e_2 \in E(T)$, there exist an $(r-1)$-element subset $f \subseteq e_2$, an edge $e_1 \neq e_2$ containing $f$, and two edge-disjoint subtrees $T_1$ and $T_2$ of $T$ such that 
    \[V(T_1) \cap V(T_2) = f,  \qquad E(T_1) \cup E(T_2) = E(T),\]
    \[e_1 \in E(T_1), \qquad  \text{and} \qquad  e_2 \in E(T_2).\]
\end{lemma}
\begin{proof}
    We prove this by induction on $k$. If $k = 2$, then we can take $e_1$ to be the other edge in $T$ besides $e_2$, take $f := e_1 \cap e_2$, and take $T_i := T[e_i]$ for each $i \in \{1,2\}$.

    Now assume $k \geq 3$. Let $(g_1, \ldots, g_k)$ be a tight tree ordering of $E(T)$, and for each $2 \leq j \leq k$, let $p(j) < j$ be an index such that $|g_j \cap g_{p(j)}| = |g_j \cap \bigcup_{i < j} g_i| = r-1$. If $e_2 = g_k$, then we can take 
    \[e_1 := g_{p(k)}, \quad f := g_k \cap g_{p(k)},\quad  T_1 := T\left[\bigcup_{i < k} g_i\right], \quad \text{and} \quad T_2 := T[g_k].\] Otherwise, we apply the inductive hypothesis to $T' := T[\bigcup_{i< k} g_i]$ to obtain $f$, $e_1$, $T_1'$, and $T_2'$ (with $T_1'$ and $T_2'$ playing the respective roles of $T_1$ and $T_2$). Assume without loss of generality that $g_{p(k)} \in E(T_1')$. We now add~$g_k$ to~$T_1'$ to obtain $T_1$, and we take $T_2 := T_2'$.  
    Now because 
    \[g_k \cap V(T_2') \subseteq g_{p(k)} \cap V(T_2') \subseteq V(T_1') \cap V(T_2') = f,\]
    it follows that $V(T_1) \cap V(T_2) = f$.
    All other required properties of $T_1$ and $T_2$ are clearly inherited from $T_1'$ and $T_2'$.
\end{proof}

Now given such subtrees $T_1$ and $T_2$ as in \Cref{lem:tight tree splitting}, the following lemma, combined with \Cref{lem:reordering lemma}, gives a bound on the expected number of $\pi$-bad edges for $(T, \vec e)$, which we can use in our inductive step in proving \Cref{lem:main hypergraph}.

\begin{lemma}\label{lem:gluing random hypergraph}
Let $\pi \in \mathcal S_H$ be a uniformly random ordering of $V(H)$. Let $T_1$ and $T_2$ be edge-disjoint trees whose vertex sets intersect in an $(r-1)$-element set $f = \{u_1, \ldots, u_{r-1}\}$. Further assume that for each $i \in \{1,2\}$, there exists an edge $e_i \in E(T_i)$ containing $f$ and another vertex $z_i$. Consider the ordering $\vec e_i := (u_1, \ldots, u_{r-1}, z_i)$ of $e_i$ for $i \in \{1,2\}$, and let $T := T_1 \cup T_2$. Then
\[\mathbb E[\bad_{\pi}(T, \vec e_2)] \leq \mathbb E[\bad_\pi (T_1, \vec e_1)] + \mathbb E[\bad_\pi (T_2, \vec e_2)] + 1.\]
\end{lemma}
\begin{proof}
    We define a bijection $\varphi : \mathcal S_H \to \mathcal S_H$ as follows. Given $\pi = (v_1, \ldots, v_n) \in \mathcal S_H$, let $j = j_{\pi}$ be the minimum index with $r \leq j \leq n$ such that $\{v_1, \ldots, v_{r-1}, v_j\}$ is an edge of $H$ spanning a copy of $(T_1, \vec e_1)$ in $\pi$, if such an index exists. Otherwise, define $j_{\pi} := n$. 
   We then define 
   \[\varphi(\pi) := (v_1, \ldots, v_{r-1}, v_{j+1}, \ldots, v_n, v_j, \ldots, v_r).\]
   That is, $\varphi(\pi)$ is obtained from $\pi$ by reversing the segment $(v_r, \ldots, v_j)$ and moving it to the end of the sequence.
To see that $\varphi$ is bijective, it suffices to establish injectivity. To this end, suppose that $\varphi(\pi) = \varphi(\pi')$ for some distinct $\pi = (v_1, \ldots, v_n)$ and $\pi' = (v'_1, \ldots, v'_n)$. Let $j := j_\pi$ and $j' := j'_\pi$. Since $\pi \neq \pi'$, we first note that $j \neq j'$, so assume without loss of generality that $j' < j$. Then $(v_1, \ldots, v_{j'}) = (v'_1, \ldots, v'_{j'})$. Since $j' = j_{\pi'} < n$, we know that $\{v'_1, \ldots, v'_{r-1}, v'_{j'}\} = \{v_1, \ldots, v_{r-1}, v_{j'}\}$ spans a copy of $(T_1, \vec e_1)$ in $\pi'$. But then the same edge spans a copy of $(T_1, \vec e_1)$ in $\pi$, and this contradicts the definition of $j = j_\pi$. It follows that $\varphi$ is inective, hence bijective.

Now consider a fixed ordering $\pi = (v_1, \ldots, v_n) \in \mathcal S_H$, and let $j = j_{\pi}$. We wish to bound $\bad_{\pi}(T, \vec e_2)$ in terms of $\bad_{\pi}(T_1, \vec e_1)$ and $\bad_{\varphi(\pi)}(T_2, \vec e_2)$. The key observation is that if $j < n$, we know that $\{v_1, \ldots, v_{r-1}, v_j\}$ spans a copy of $(T_1, r, e_1)$ in $\pi$, so every edge of $H$ of the form $\{v_1, \ldots, v_{r-1}, v_i\}$ with $i > j$ which spans a copy of $(T_2, \vec e_2)$ in $\varphi(\pi)$ must also span a copy of $(T, \vec e_2)$ in~$\pi$. Indeed, the respective embeddings of $T_1$ and~$T_2$ into $H[\{v_1, \ldots, v_j\}]$ and $H[\{v_1, \ldots, v_{r-1}, v_{j+1}, \ldots, v_i\}]$ agree on $V(T_1) \cap V(T_2) = f$, mapping~$u_\ell$ to~$v_\ell$ for all $\ell \in [r-1]$, so they combine to give an embedding of $T$ into $H[V_i(\pi)]$ which additionally maps $z_2$ to $v_i$. Thus the number of edges of $H$ the form $\{v_1, \ldots, v_{r-1}, v_i\}$ with $i > j$ which are $\pi$-bad for $(T, \vec e_2)$ is at most $\bad_{\varphi(\pi)}(T_2, \vec e_2)$ (and this holds even if $j = n$). On the other hand, by the definition of $j = j_{\pi}$, every edge of $H$ of the form $\{v_1, \ldots, v_{r-1}, v_i\}$ with $i < j$ is $\pi$-bad for $(T_1, \vec e_1)$, so the total number of such edges of $H$ is at most $\bad_{\pi}(T_1, \vec e_1)$. The only other index that can contribute to $\bad_{\pi}(T, \vec e_2)$ is $j$ itself, so we obtain
\[\bad_{\pi}(T, \vec e_2) \leq \bad_{\pi}(T_1, \vec e_1) + \bad_{\varphi(\pi)}(T_2, \vec e_2) + 1.\]

Now let $\pi$ be uniformly random in $\mathcal S_H$ once again. Since $\varphi$ is a bijection $\mathcal S_H \to \mathcal S_H$, we have that $\varphi(\pi)$ has the same (uniform) distribution as $\pi$, which means that
\[\mathbb E[\bad_{\pi}(T, \vec e_2)] \leq \mathbb E[\bad_{\pi}(T_1, \vec e_1)] + \mathbb E[\bad_{\pi}(T_2, \vec e_2)] + 1\]
by the linearity of expectation.
\end{proof}

\subsection{Completing the proof}
We now finish the proof of \Cref{thm:main hypergraph} by proving \Cref{lem:main hypergraph}.
\begin{proof}[Proof of \Cref{lem:main hypergraph}]
    We prove the lemma by induction on $e(T)$. If $e(T) = 1$, then for every $\pi = (v_1, \ldots, v_n) \in \mathcal S_H$, every $\pi$-marked edge of $H$ spans a copy of $(T, \vec e)$, so $\bad_\pi(T, \vec e) \equiv 0$.

    Now assume $e(T) \geq 2$. Let $e_2$ be the underlying edge of $T$ ordered by $\vec e$, and let $f$, $e_1$, $T_1$, and $T_2$ be as in \Cref{lem:tight tree splitting}. Let $u_1, \ldots, u_{r-1}, z_1, z_2 \in V(T)$ be such that $f = \{u_1, \ldots, u_{r-1}\}$ and $e_i = f \cup \{z_i\}$ for each $i \in \{1,2\}$. Consider the ordering $\vec e_i = (u_1, \ldots, u_{r-1}, z_i)$ of $e_i$ for each $i \in \{1,2\}$. Since $\vec e_2$ and $\vec e$ are orderings of the same edge $e_2 \in E(T)$, we have by \Cref{lem:reordering lemma} that
    \[\mathbb E[\bad_\pi(T, \vec e)] = \mathbb E[\bad_\pi (T, \vec e_2)].\]
    Therefore, by \Cref{lem:gluing random hypergraph} and the inductive hypothesis, we have
    \begin{align*}
        \mathbb E[\bad_\pi(T, \vec e)] &\leq \mathbb E[\bad_\pi(T_1, \vec e_1)] + \mathbb E[\bad_\pi(T_2, \vec e_2)] + 1 \\
        &\leq (e(T_1)-1) + (e(T_2)-1) + 1 \\
        &= e(T) - 1,
    \end{align*}
    and this completes the proof.
\end{proof}

\section{Extremal cases for Kalai's conjecture}\label{sec:extremal cases}

This section is devoted to proving \Cref{thm:extremal cases hypergraph}. Throughout this section, $T$ will be a fixed tight tree with $k \geq 1$ edges, and $H$ will be a fixed $n$-vertex, $r$-uniform hypergraph with exactly $\frac{k-1}{r} \binom{n}{r-1}$ edges and with no copy of $T$. We also keep the same notation defined in \Cref{sec:notation}, with respect to this hypergraph $H$.

Since
\Cref{thm:extremal cases hypergraph} is trivial for $k=1$, we will assume that $k \geq 2$. Let $x$ be a leaf of $T$ (i.e. a vertex contained in a single edge of $T$), let $e' = \{u_1, \ldots, u_{r-1}, x\}$ be the edge of $T$ containing $x$, and let $e \neq e'$ be an edge of $T$ containing $u_1, \ldots, u_{r-1}$ and some other vertex $y$. Fix the ordering $\vec e = (u_1, \ldots, u_{r-1},y)$ of $e$. Let $T'$ be the tight tree obtained from $T$ by deleting the vertex $x$ and the edge $e'$.

If $T$ is not a sunflower, then $e'$ cannot be adjacent to every vertex in $J(T)$, so we can choose~$e$ so that there exists another edge $e''$ of $E(T)$ which contains $u_1, \ldots, u_{r-2}, y$ and some additional vertex $z$, but not $u_{r-1},x$.

The proof is broken up into several lemmas. The first two use \Cref{lem:main hypergraph} as a black box to deduce that $H$ is an $(r-1)$-$(n,r,k-1)$-design, meaning that $\deg_H(S) = k-1$ for every set $S$ of $r-1$ vertices of $H$.

\begin{lemma}\label{lem:one spanning edge}
     For every ordering $\pi \in \mathcal S_H$, the number of $\pi$-marked edges of $H$ which span a copy of $(T', \vec e)$ in $\pi$ is exactly one.
\end{lemma}
\begin{proof}
    First, note that for every $\pi = (v_1, \ldots, v_n) \in \mathcal S_H$, the number of $\pi$-marked edges of $H$ which span a copy of $(T', \vec e)$ in $\pi$ cannot be more than one. Indeed, if $\{v_1, \ldots, v_{r-1}, v_j\}$ and $\{v_1, \ldots, v_{r-1}, v_{j'}\}$ were two distinct such $\pi$-marked edges with $r \leq j < j' \leq n$, then the embedding of $T'$ into $H[V_j(\pi)]$ mapping $u_i$ to $v_i$ for all $i \in [r-1]$ would extend to an embedding of $T$ into $H$ by designating $x \mapsto v_{j'}$. But we assumed that $H$ had no copy of $T$, so this cannot happen.

    Second, note that for a uniformly random $\pi \in \mathcal S_H$, because $H$ has exactly $\frac{k-1}{r} \binom{n}{r-1}$ edges, the expected number of $\pi$-marked edges of $H$ spanning a copy of $(T', \vec e)$ in $\pi$ is
    \[(k-1) - \mathbb E[\bad_\pi(T', \vec e)] \geq (k-1) - (e(T')-1) = 1\]
    by \Cref{lem:main hypergraph}. These two facts together imply the lemma.
\end{proof}

\begin{lemma}\label{lem:design}
    The hypergraph $H$ is an $(r-1)$-$(n,r,k-1)$-design.
\end{lemma}
\begin{proof}
    Suppose by way of contradiction that some $(r-1)$-element subset $S = \{v_1, \ldots, v_{r-1}\}$ of $V(H)$ is contained in at least $k$ edges of $H$. Now consider any ordering $\pi = (v_1, \ldots, v_n) \in \mathcal S_H$ which places $v_1, \ldots, v_{r-1}$ in the first $r-1$ coordinates. Note that there are at least $k$ $\pi$-marked edges. By \Cref{lem:one spanning edge}, there exists some $\pi$-marked edge of $H$ spanning a copy of $(T', \vec e)$ in $\pi$. In particular, there exists an embedding of~$T'$ into $H$ mapping $u_i$ to $v_i$ for all $i \in [r-1]$. Since~$T'$ has $k-1$ edges, it has exactly $k-1$ vertices aside from $u_1, \ldots, u_{r-1}$, so there exists some $\pi$-marked edge $\{v_1, \ldots, v_{r-1}, v_j\}$ such that this embedding misses~$v_j$. Now assigning $x \mapsto v_j$ extends this to an embedding of $T$ into $H$, which contradicts that $H$ has no copy of~$T$.

    We conclude that every $\deg_H(S) \leq k-1$ for every $(r-1)$-element subset of $V(H)$. On the other hand, because $H$ has exactly $\frac{k-1}{r} \binom{n}{r-1}$ edges, the expectation of $\deg_H(S)$ when $S \in \binom{V(H)}{r-1}$ is chosen uniformly at random is exactly $k-1$. Thus we must have $\deg_H(S) = k-1$ for every $S \in \binom{V(H)}{r-1}$, which makes $H$ an $(r-1)$-$(n,r,k-1)$-design.
\end{proof}

This completes the proof of \Cref{thm:extremal cases hypergraph} in the case when $T$ is a sunflower. Henceforth, we will assume that $T$ is not a sunflower, and so the edge $e''$ is well-defined. We will use the following lemma to embed a copy of $T$ whenever $H$ is not of the form specified by \Cref{thm:extremal cases hypergraph}.
\begin{lemma}\label{lem:tight tree ordering}
    There exists a tight tree ordering $(e_1, \ldots, e_{k-1})$ of $E(T')$ with $e_1 = e$ and $e_2 = e''$.
\end{lemma}
\begin{proof}
    Since $T'$ is a tight tree, we know that $J(T')$ is connected. Consider a breadth-first search in $J(T')$ from the vertex $e$ which visits its neighbor~$e''$ first, and let $e = e_1, e'' = e_2, \ldots, e_{k-1}$ be the order in which vertices of $J(T')$ were visited. Then for each $2 \leq j \leq k-1$, we have that $e_j$ intersects some previous edge $e_i \in E(T')$ with $i < j$ in $r-1$ vertices. Also, we have $\left| \bigcup_{j=1}^{k-1} e_j\right| = |V(T')| = r+k-2$, so $(e_1, \ldots, e_{k-1})$ is indeed a tight tree ordering. \qedhere
\end{proof}

\begin{proof}[Proof of \Cref{thm:extremal cases hypergraph}]
Again, we will assume that $T$ is not a sunflower, as this case is covered by \Cref{lem:design}. Let $H'$ be a tight component of $H$. We will show that $H'$ is a clique. Suppose not. Then there exist $r+1$ vertices $v_1, \ldots, v_{r-1}, y', z' \in V(H')$ such that $\{v_1, \ldots, v_{r-1}, y'\}$ and $\{v_1, \ldots, v_{r-2}, y', z'\}$ are edges in $H'$ but $\{v_1, \ldots, v_{r-1}, z'\}$ is not.

Now we will try to create an embedding $\Phi : T \to H'$, step by step. Let $(e_1, \ldots, e_{k-1})$ be the tight tree ordering of $E(T')$ given by \Cref{lem:tight tree ordering}. We first designate $\Phi(u_i) = v_i$ for $i \in [r-1]$, $\Phi(y) = y'$, and $\Phi(z) = z'$. This embeds $e = e_1$ and $e'' = e_2$. Next, for each $j = 3, \ldots, k-1$, in turn, we may embed $e_j$ greedily as follows. Suppose we have embedded $e_1, \ldots, e_{j-1}$. Write $e_j = f_j \cup \{x_j\}$, where $f_j$ contains the $r-1$ previously embedded vertices. By \Cref{lem:design}, $\Phi(f_j)$ is contained in $k-1$ edges of $H'$, and there are only $j-1 \leq k-2$ vertices of $H'$ in the image of $\Phi$ but outside $\Phi(f_j)$, so there exists a fresh choice for $\Phi(x_j)$ so that $\Phi(f_j) \cup \{\Phi(x_j)\} \in E(H)$. So far, this embeds $T'$ into $H'$. It remains to embed the leaf vertex $x$, which is contained in the edge $e' = \{u_1, \ldots, u_{r-1}, x\} \in E(T)$.

For this, note that because $\{v_1, \ldots, v_{r-1}, z'\}$ is not an edge of $H'$, there are at most $k-2$ vertices $x'$ in~$\Phi(T')$ such that $\{\Phi(u_1), \ldots, \Phi(u_{r-1}), x'\} = \{v_1, \ldots, v_{r-1}, x'\} \in E(H')$. Therefore, since $\{v_1, \ldots, v_{r-1}\}$ is contained in $k-1$ edges of $H'$ by \Cref{lem:design}, there exists a choice of $x'$ outside $\Phi(T')$ such that $\{v_1, \ldots, v_{r-1}, x'\} \in E(H')$. Now assigning $\Phi(x) = x'$ completes the embedding $\Phi : T \to H'$, which contradicts that $H$ has no copy of $T$.

Thus every tight component of $H$ is a clique. Call a vertex subset of $H$ supporting some tight component of $H$ a \emph{block}. Then $E(H)$ consists of all $r$-element sets contained in some block, and no two distinct blocks intersect in more than $r-2$ vertices.  
Moreover, by \Cref{lem:design}, every set of $r-1$ vertices of $H$ is contained in $k-1$ edges, and so is contained in exactly one block, which must have size $k + r - 2$. Thus the hypergraph on $V(H)$ whose edges are the blocks form an $S(r-1, k+r-2, n)$, as desired.
\end{proof}

\section{Concluding remarks}\label{sec:concluding remarks}
With some simple modifications, the same proof for Kalai's conjecture (\Cref{thm:main hypergraph}) given in \Cref{sec:kalai} can be strengthened to give a common generalization of \Cref{thm:main directed} for digraphs and \Cref{thm:main hypergraph} for hypergraphs.

To state the generalization, we establish some more terminology. An $r$-uniform \emph{fully directed hypergraph}\footnote{There are multiple ways of defining a directed $r$-uniform hypergraph which converge to the standard digraph definition when $r=2$. The \emph{fully} directed hypergraphs described here are not the most common such notion, but they are quite natural, and they have been considered in~\cite{brown1984digraph, cameron2016extremal, DEVINE2027115410}.}~$\vec H$ consists of a finite set $V(\vec H)$ of \emph{vertices} and a set $E(\vec H)$ of ordered $r$-tuples of distinct vertices of $H$, called \emph{hyperarcs}. We allow multiple hyperarcs on the same set of $r$ vertices, provided that each is ordered uniquely. We say that $\vec H$ is a \emph{full orientation} of a hypergraph $H$ if $V(\vec H) = V(H)$, and for every edge $e \in E(H)$, there is a unique hyperarc $(v_1, \ldots, v_r) \in E(\vec H)$ with $\{v_1, \ldots, v_r\} = e$, and all hyperarcs are of this form. In this case, $\vec H$ is called an \emph{fully oriented hypergraph}. We say that a fully oriented hypergraph $\vec H$ is \emph{antidirected} if there exists an $r$-partition $(V_1, \ldots, V_r)$ of $V(\vec H)$ such that $E(\vec H) \subseteq V_1 \times \cdots \times V_r$.
\begin{theorem}\label{thm:generalized}
    Let $k \geq 1$ be an integer, and let $\vec H$ be an $n$-vertex, $r$-uniform fully directed hypergraph with more than $(k-1)(r-1)!\binom{n}{r-1}$ hyperarcs. Then $\vec H$ contains every antidirected full orientation of every $r$-uniform tight tree with $k$ edges.
\end{theorem}

Clearly, taking $r=2$ recovers \Cref{thm:main directed}. To see how \Cref{thm:main hypergraph} follows, simply replace every edge of the host $r$-uniform hypergraph $H$ with all $r!$ possible ordered $r$-tuples on the same set of vertices to obtain a strongly directed hypergraph $\vec H$. If $H$ has more than $\frac{k-1}{r} \binom{n}{r-1}$ edges, then $\vec H$ has more than $(k-1)(r-1)! \binom{n}{r-1}$ hyperarcs, and so contains a copy of every antidirected full orientation of every tight tree with $k$ edges by \Cref{thm:generalized}, which correspond to copies of tight trees in $H$.

We do not give the full details of the proof of \Cref{thm:generalized}, but we instead outline how to modify the proof \Cref{thm:main hypergraph} to this setting. Primarily, we must define a separate variant of $\bad_\pi(\vec T, \vec e)$ for every relative orientation of a hyperarc when compared to $\pi$, analogous to how we separately defined $\bad^{\textrm{out}}(T,r,e)$ and $\bad^{\textrm{in}}(T,r,e)$ in \Cref{sec:antidirected}.

Specifically, if $\pi = (v_1, \ldots, v_n)$ is a ordering of $V(\vec H)$ and $\tau : [r] \to [r]$ is a permutation of $[r]$, a hyperarc of $\vec H$ is said to be $(\pi, \tau)$-marked if it has the form $\tau(v_1, \ldots, v_{r-1}, v_j)$ for some $r \leq j \leq n$.
If $\vec T$ is an oriented tight tree, and $\vec e = (u_1, \ldots, u_r)$ is a reordering of a hyperarc $\tau(\vec e) = (u_{\tau^{-1}(1)}, \ldots, u_{\tau^{-1}(r)}) \in E(\vec T)$, we say that a $(\pi,\tau)$-marked hyperarc $\tau(v_1, \ldots, v_{r-1}, v_j) \in E(\vec H)$ \emph{spans} a copy of $(\vec T, \vec e)$ in $\pi$ if there is an embedding of $\vec T$ into $\vec H[V_j(\pi)]$ mapping $u_i$ to $v_i$ for all $i \in [r-1]$ and mapping $u_r$ to $v_j$. Define $\bad_\pi^{\tau}(\vec T, \vec e)$ to be the number of $(\pi, \tau)$-marked hyperarcs of $\vec H$ which \emph{do not} span a copy of $(\vec T, \vec e)$ in $\pi$.

Given a uniformly random ordering $\pi$ of $V(H)$, our argument from \Cref{lem:reordering lemma} gives
\[\mathbb E[\bad^\tau_\pi(\vec T, \vec e)] = \mathbb E[\bad^{\tau \sigma^{-1}}_\pi(\vec T, \sigma(\vec e))]\]
for every pair $\sigma, \tau$ of permutations of $[r]$, every fully oriented tight tree $\vec T$, and every hyperarc $\tau(\vec e) \in E(\vec T)$. The argument from \Cref{lem:gluing random hypergraph} shows that for any pair $\vec T_1$, $\vec T_2$ of hyperarc-disjoint fully oriented tight trees and hyperarcs $\tau(\vec e_1) = \tau(u_1, \ldots, u_{r-1}, z_1) \in E(\vec T_1)$ and $\tau(\vec e_2) = \tau(u_1, \ldots, u_{r-1}, z_2) \in E(\vec T_2)$ such that $V(\vec T_1) \cap V(\vec T_2) = \{u_1, \ldots, u_{r-1}\}$, we have
\[\mathbb E[\bad_{\pi}^{\tau}(\vec T_1 \cup \vec T_2, \vec e_2)] \leq \mathbb E[\bad_\pi^\tau (\vec T_1, \vec e_1)] + \mathbb E[\bad_\pi^\tau (\vec T_2, \vec e_2)] + 1.\]
\Cref{lem:tight tree splitting} then allows us to put these two facts together to inductively show, as in the proof of \Cref{lem:main hypergraph}, that whenever $\tau(\vec e) \in E(\vec T)$, we have 
\[\mathbb E[\bad_\pi^{\tau}(\vec T, \vec e)] \leq e(\vec T)-1,\]
and \Cref{thm:generalized} follows.
\vspace{10pt}

\noindent
\textbf{Statement of AI use.} The proofs herein were developed and written by traditional means without the use of artificial intelligence. The figures were also created by the author. Artificial intelligence (GPT-6 Astra) was used only to assist with literature search.

\vspace{10pt}

\noindent
\textbf{Acknowledgements.} We would like to thank Liana Yepremyan for looking over our manuscript and providing helpful feedback.

\bibliographystyle{abbrv}
\bibliography{Erdos-Sos-References}

\end{document}